\documentclass[10pt,a4paper,notitlepage, oneside]{article}
\usepackage[utf8x]{inputenc} 
\usepackage{ucs} 
\usepackage[english]{babel}
\usepackage{amsmath}
\usepackage{amsthm}
\usepackage{amsfonts}

\usepackage{amssymb} 
\usepackage{mathtools}
\usepackage{cite}
\usepackage{xfrac} 
\usepackage[percent]{overpic}
\usepackage{hyperref}
\hypersetup{
	colorlinks=true,  
	linkcolor=black,    
	urlcolor=red      
}

\usepackage{enumerate} 

\usepackage{tikz}
\usetikzlibrary{backgrounds}
\usetikzlibrary{decorations.pathmorphing}

\usetikzlibrary{arrows,positioning,intersections}
\usetikzlibrary{backgrounds}
\usetikzlibrary{trees}
\usetikzlibrary{snakes}

\tikzstyle{dedge}=[thick, densely dotted]
\tikzstyle{novertex}=[rectangle]
\tikzstyle{hvertex}=[circle,inner sep=0.cm, minimum size=1mm, fill=white, draw=black]
\tikzstyle{hedge}=[ thick]
\tikzstyle{harc}=[ultra thick, ->]
\tikzstyle{point}=[draw,circle,inner sep=0.cm, minimum size=1mm, fill=black]
\tikzstyle{face}=[color=auchblau,fill=hellblau,thick]
\tikzstyle{nface}=[color=hellblau,fill=hellblau,thick] 

\colorlet{auchblau}{blue!60!white}

\colorlet{hellblau}{blue!20!white}
\colorlet{hellrot}{red!40!white}
\colorlet{hellgrau}{black!30!white}

\tikzstyle{pedge}=[ ->, thick]

\usepackage[textsize = tiny]{todonotes}  

\newtheorem{definition}{Definition}
\newtheorem{observation}[definition]{Observation}

\newtheorem{theorem}[definition]{Theorem}
\newtheorem{corollary}[definition]{Corollary}
\newtheorem{lemma}[definition]{Lemma}
\newtheorem{conjecture}[definition]{Conjecture}

\title{Cycle decompositions of pathwidth-$6$ graphs}
\author{Elke Fuchs, Laura Gellert and Irene Heinrich}
\date{}
\begin{document}

\maketitle 
	
\begin{abstract}
Hajós' conjecture asserts that a simple Eulerian graph on $n$ vertices can be decomposed into at most $\left\lfloor(n − 1)/2 \right\rfloor$ cycles. The conjecture is only proved for graph classes in which every element contains vertices of degree $2$ or $4$. We develop new techniques to construct cycle decompositions. They work on  the common neighbourhood of two degree-$6$ vertices. 
With these techniques we find structures that cannot occur in a minimal counterexample to Haj\'os' conjecture and verify the conjecture for  Eulerian graphs  of pathwidth at most $6$. This implies that these graphs satisfy the \emph{small cycle double cover conjecture}. 
\end{abstract}

\section{Introduction}

It is well-known that the edge set of an Eulerian graph can be decomposed into cycles. In this context, a natural question arises: How many cycles are needed to decompose the edge set of an Eulerian graph? Clearly, a graph $G$ with a vertex of degree $|V(G)|-1$ cannot be decomposed into less than $\left\lfloor\sfrac{\left(|V(G)|-1\right)}{2}\right\rfloor$ many cycles. Thus, for a general graph $G$, we cannot expect to find a cycle decomposition with less than $\left\lfloor \sfrac{\left(|V(G)|-1\right)}{2} \right\rfloor$ many cycles.
Haj\'os' conjectured that this number of cycles will always suffice.\footnote{Originally, Haj\'os' conjectured a bound of $\lfloor\sfrac{|V(G)|}{2}\rfloor$. Dean~\cite{Dean86} showed that Haj\'os' conjecture is equivalent to the conjecture with bound $\left\lfloor \sfrac{(|V(G)| -1)}{2}\right\rfloor$. }

\begin{conjecture}[Haj\'os' conjecture (see~\cite{Lovasz68})] 
Every simple Eulerian graph $G$ has a cycle decomposition with at most $\left\lfloor \sfrac{(|V(G)| -1)}{2}\right\rfloor$ many cycles.
\end{conjecture}

 Granville and Moisiadis~\cite{Gran87} showed that for every $n \geq 3$ and every $i \in \left\lbrace 1, \ldots, \lfloor\sfrac{(|V(G)|-1)}{2}\rfloor \right\rbrace$, there exists a connected graph with $n$ vertices and maxi\-mum degree at most $4$ whose minimal cycle decomposition consists of exactly~$i$ cycles. This shows that --- even if the maximal degree is restricted to $4$ --- the bound $ \left\lfloor \sfrac{(|V(G)|-1)}{2} \right\rfloor$ is best possible. 
 
 A simple lower bound on the minimal number of necessary cycles is the maximum degree divided by $2$. This bound is achieved by the complete bipartite graph $K_{2k,2k}$ that can be decomposed into $k$ Hamiltonian cycles (see~\cite{Las_Au76}). In general, all graphs with a Hamilton decomposition (for example complete graphs $K_{2k+1}$~\cite{alspach2008wonderful}) trivially satisfy Haj\'os' conjecture.

Haj\'os' conjecture remains wide open for most classes.
Heinrich, Natale and Streicher \cite{Heinrich_Streicher_17} verified Haj\'os' conjecture for small graphs by exploiting Lemma~\ref{lem:uv_adj_5common_nb_forbidden_structures}, \ref{lem:nonadj_6_common_nb_forbidden}, \ref{lem:nbhd_deg4_vertex}, and  \ref{lem:deg6_neighbourhood} of this paper as well as random heuristics and integer programming techniques: 
\begin{theorem}[Heinrich, Natale and Streicher \cite{Heinrich_Streicher_17}]\label{thm:Hajos_12vertices}
Every simple Eulerian graph with at most $12$ vertices satisfies Hajós' Conjecture.
\end{theorem}
Apart from Hamilton decomposable (and small) graphs, the conjecture has (to our know\-ledge) only been shown for graph classes in which every element contains vertices of degree at most $4$.  
Granville and Moisiadis~\cite{Gran87} showed that Hajós' conjecture is satisfied for all Eulerian graphs with maximum degree at most $4$.
Fan and Xu~\cite{Fan_Xu02} showed that all Eulerian graphs that are embeddable in the projective plane or do not contain the minor $K_6^-$ satisfy Haj\'os' conjecture. 
To show this, they provided four operations involving vertices of degree less than $6$ that transform an Eulerian graph not satisfying Haj\'os' conjecture into another Eulerian graph not satisfying the conjecture that contains at most one vertex of degree less than $6$. This statement generalises the work of Granville and Moisiadis~\cite{Gran87}. As all four operations preserve planarity, the statement further implies that planar graphs satisfy Haj\'os' conjecture. This was shown by Seyffarth~\cite{seyffarth1992hajos} before. The conjecture is still open for toroidal graphs. Xu and Wang~\cite{Xu_Wang05} showed that the edge set of each Eulerian graph that can be embedded on the torus can be decomposed into at most $ \left\lfloor\sfrac{(|V(G)|+ 3)}{2} \right\rfloor$ cycles. Heinrich and Krumke~\cite{Heinrich_Krumke_17} introduced a linear time procedure that computes minimum cycle decompositions in treewidth-$2$ graphs of maximum degree $4$.

We contribute to the sparse list of graph classes satisfying Haj\'os' conjecture. Our class contains graphs without any vertex of degree $2$ or $4$ --- in contrast to the above mentioned graph classes. 
\begin{theorem} \label{thm:mainthm_pw6}
Every Eulerian graph $G$ of pathwidth at most $6$ satisfies Haj\'os' conjecture.
\end{theorem}
As graphs of pathwidth at most $5$  contain two vertices of degree less than~$6$, it suffices to concentrate on graphs of pathwidth exactly $6$. All such graphs with at most one vertex of degree $2$ or $4$ contain two degree-$6$ vertices that are either non-adjacent with the same neighbourhood or adjacent with four or five common neighbours. We use these structures to construct cycle decompositions. 

With similar ideas, it is possible attack graphs of treewidth $6$. As more substructures may occur, we restrict ourselves to graphs of pathwidth~$6$.

\bigskip

A \emph{cycle double cover} of a graph $G$ is a collection $\mathcal{C}$ of cycles of $G$ such that each edge of $G$ is contained in exactly two elements of $\mathcal{C}$.
 The popular \emph{cycle double cover conjecture} asserts that every $2$-edge connected graph admits a cycle double cover. This conjecture is trivially satisfied for Eulerian graphs. 
Haj\'os' conjecture implies a conjecture of Bondy regarding the Cycle double cover conjecture. 
\begin{conjecture}[Small Cycle Double Cover Conjecture (Bondy~\cite{Bondy90})] 
Every simple $2$-edge connected graph $G$ admits a cycle double cover of at most $|V(G)|-1$ many cycles. 
\end{conjecture}
As a cycle double cover may contain a cycle twice, we can conclude the following directly from Theorem~\ref{thm:mainthm_pw6}.
\begin{corollary}
Every Eulerian graph $G$ of pathwidth at most $6$  satisfies the small cycle double cover conjecture.
\end{corollary}

\section{Reducible structures}
All graphs considered in this paper are finite, simple and Eulerian. 
We use standard graph theory notation as can be found in the book of Diestel~\cite{Die}.

In order to prove our main theorem, we consider a cycle decomposition of a graph $G$ as a colouring of the edges of $G$ where each colour class is a cycle. 	
We define a \emph{legal colouring} $c$ of a graph $G$ as a map 
$$c:E(G) \mapsto \left\lbrace  1, \ldots, \left\lfloor \sfrac{(|V(G)|-1)}{2} \right\rfloor \right\rbrace$$ where each colour class $c^{-1}(i)$ for $i \in  \left\lbrace  1, \ldots, \left\lfloor \sfrac{(|V(G)|-1)}{2} \right\rfloor \right\rbrace $ is the edge set of a cycle of $G$. A legal colouring is thus associated to a cycle decomposition of $G$ that satisfies Haj\'os' conjecture. 

Using recolouring techniques, we show the following lemmas for two degree-$6$ vertices with common neighbourhood $N$ of size $4$, $5$ or $6$. 
All proofs can be found in Section~\ref{sec:proofs_thms}.

\begin{lemma} \label{lem:uv_adj_5common_nb_forbidden_structures}
Let $G$ be an Eulerian graph with two degree-$6$ vertices $u,v$ with
\begin{equation*}
N(u)= N \cup \lbrace  v \rbrace \qquad N(v)= N \cup \lbrace  u \rbrace .
\end{equation*}
Let all Eulerian graphs obtained from $G - \lbrace u, v \rbrace $ by addition or deletion of edges with both end vertices in $N$ have a legal colouring. \\
If $G[N]$  contains at least one edge, or if   $G - \lbrace u, v \rbrace $ contains a vertex that is adjacent to at least three vertices of $N$ then $G$ also has a legal colouring. 
\end{lemma}

\begin{lemma} \label{lem:uv_adj_4common_nb_forbidden_structures}
Let $G$ be an Eulerian graph with two  degree-$6$ vertices $u,v$ with 
\begin{equation*}
N(u)=  N \cup \lbrace  u, x_v  \rbrace \qquad N(v)= N \cup \lbrace  v, x_u  \rbrace.
\end{equation*}
Let $P$ be an $x_u$-$x_v$-path in $G - \lbrace u, v \rbrace -N$.  Further let all Eulerian graphs obtained from $G - \lbrace u, v \rbrace $ by addition and deletion of edges with both end vertices in $N \cup \lbrace x_u,x_v \rbrace$  and by optional deletion of $E(P)$ have a legal colouring. \\
If  $G[N\cup\{x_u, x_v\}]$  contains at least one edge not equal to $x_ux_v$, or if $G - \lbrace u, v \rbrace $ contains a vertex that is adjacent to at least three vertices of $N$
then $G$ also has a legal colouring. 
\end{lemma}

\begin{lemma}\label{lem:nonadj_6_common_nb_forbidden} 
Let $G$ be an Eulerian graph with two degree-$6$ vertices $u,v$ with
\begin{equation*}
N(u)= N(v)=N.
\end{equation*}
Let all Eulerian graphs obtained from $G - \lbrace u, v \rbrace $ by addition or deletion of edges with both end vertices in $N$ have a legal colouring. \\
If $G[N]$  contains at least one edge, or if   $G - \lbrace u, v \rbrace $ contains a vertex that is adjacent to at least three vertices of $N$ then $G$ also has a legal colouring. 
\end{lemma}

The next two results are not necessary for the proof of Theorem~\ref{thm:mainthm_pw6}. We nevertheless state them here.

The first  lemma is useful for graphs with an odd number of vertices.

\begin{lemma} \label{lem:deg2_or_4_neighbourhood_odd}
Let $G$ be an Eulerian graph on an odd number $n$ of vertices that contains a vertex $u$ of degree $2$ or $4$ with neighbourhood $N$.
Let $G'$ be obtained from $G - \lbrace u \rbrace $ by addition or deletion of arbitrary edges in $G[N]$. 
If $G'$ has a legal colouring, then $G$ has a legal colouring. 
\end{lemma}

If a graph $G$ contains a degree-$2$ vertex $v$ with independent neighbours $x_1,x_2$, then it is clear that a legal colouring of $G-v+x_1x_2$ can be transformed into a legal colouring of~$G$. 
Granville and Moisiadis~\cite{Gran87} observed a similar relation for a degree-$4$ vertex. 
	
\begin{lemma} [Granville and Moisiadis~\cite{Gran87}]  \label{lem:nbhd_deg4_vertex}
Let $G$ be an Eulerian graph containing a vertex $v$ with neighbourhood $N= \lbrace x_1, \ldots, x_4 \rbrace$ such that $G[N]$ contains the edge $x_1x_2$ but not the edge $x_3x_4$. If $G-\{vx_3,vx_4\}+\{x_3x_4\}$ has a legal colouring, then $G$ also has a legal colouring. 
\end{lemma}

Generalising this idea, we analyse the neighbourhood of a degree-$6$ vertex. 

\begin{lemma} \label{lem:deg6_neighbourhood} 
Let $G$ be an Eulerian graph that contains a degree-$6$ vertex $u$ with neighbourhood $N_G(u)=\{x_1,\ldots,x_{6}\}$ such that  $\{x_1,x_2,x_3,x_4\}$ is a clique and $x_5x_6 \notin E(G)$.
If $G'=G - \lbrace x_5u,ux_6 \rbrace + \lbrace x_5x_6 \rbrace$ has a legal colouring, then $G$ has a legal colouring. 
\end{lemma}

\section{Recolouring Techniques}
In this section, we provide recolouring techniques that are necessary to prove Lemma~\ref{lem:uv_adj_5common_nb_forbidden_structures},
\ref{lem:uv_adj_4common_nb_forbidden_structures}
and~\ref{lem:nonadj_6_common_nb_forbidden}. 
For a path $P$ or a cycle $C$ we write $c(P)=i$ or $c(C)=i$ to express that all edges of $P$ respectively $C$ are coloured with colour $i$.
We start with a statement about monochromatic triangles.
\begin{figure}[bht]
	\begin{center}
		\begin{tikzpicture}[scale=.75]
		
		\begin{scope}[shift={(-4,0)}]
		\node[hvertex] (x1) at (-1,0){};
		\node[hvertex] (x2) at (-2,1){};
		\node[hvertex] (x3) at (-2,-1){};
		\node[hvertex] (y) at (1,0){};
		
		\node[novertex] (labelx1) at (-.75,-.25){$y_1$};
		\node[novertex] (labelx2) at (-2,1.25){$y_2$};
		\node[novertex] (labelx3) at (-2,-1.25){$y_3$};
		\node[novertex] (labely) at (1.5,0){$y$};
		\node[novertex] (labelP') at (-.25,1.25){$P'$};

		\draw [hedge] (x1) to (x2);
		\draw [hedge] (x1) to (x3);
		\draw [hedge] (x2) to (x3);
		\draw [densely dotted, thick] (x1) to (y);
		
		\draw [densely dotted, thick,bend left] (x2) to (y);
		\end{scope}
		
		\begin{scope}[shift={(0,0)}]
		\node[novertex] (x) at (-1,0){};
		\node[novertex] (y) at (1,0){};
		
		\draw[pedge] (x) -- (y);
		\end{scope}

		\begin{scope}[shift={(5,0)}]
		\node[hvertex] (x1) at (-1,0){};
		\node[hvertex] (x2) at (-2,1){};
		\node[hvertex] (x3) at (-2,-1){};
		\node[hvertex] (y) at (1,0){};
		
		\node[novertex] (labelx1) at (-.75,-.25){$y_1$};
		\node[novertex] (labelx2) at (-2,1.25){$y_2$};
		\node[novertex] (labelx3) at (-2,-1.25){$y_3$};
		\node[novertex] (labely) at (1.5,0){$y$};
		\node[novertex] (labelP') at (-.25,1.25){$P'$};

		\draw [densely dotted, thick] (x1) to (x2);
		\draw [hedge] (x1) to (x3);
		\draw [hedge] (x2) to (x3);
		\draw [hedge] (x1) to (y);
		
		\draw [hedge,bend left] (x2) to (y);
		\end{scope}
		
		\begin{scope}[shift={(-4,-3.5)}]
		\node[hvertex] (x1) at (-1,0){};
		\node[hvertex] (x2) at (-2,1){};
		\node[hvertex] (x3) at (-2,-1){};
		\node[hvertex] (y) at (1,0){};
		
		\node[novertex] (labelx1) at (-.75,-.25){$x_1$};
		\node[novertex] (labelx2) at (-2,1.25){$x_2$};
		\node[novertex] (labelx3) at (-2,-1.25){$x_3$};
		\node[novertex] (labely) at (1.5,0){$y$};

		\draw [hedge] (x1) to (x2);
		\draw [hedge] (x1) to (x3);
		\draw [hedge] (x2) to (x3);
		\draw [densely dotted, thick] (x1) to (y);
		\draw [snake=zigzag,segment object length=1 pt, segment amplitude=1pt, segment length=3pt] (x2) to (y);
		\draw [snake=coil, segment aspect=2, segment amplitude=1pt,segment length=4pt] (x3) to (y);
		\end{scope}
		
		\begin{scope}[shift={(0,-3.5)}]
		\node[novertex] (x) at (-1,0){};
		\node[novertex] (y) at (1,0){};
		
		\draw[pedge] (x) -- (y);
		\end{scope}
		
		\begin{scope}[shift={(5,-3.5)}]
		\node[hvertex] (x1) at (-1,0){};
		\node[hvertex] (x2) at (-2,1){};
		\node[hvertex] (x3) at (-2,-1){};
		\node[hvertex] (y) at (1,0){};
		
		\node[novertex] (labelx1) at (-.75,-.25){$x_1$};
		\node[novertex] (labelx2) at (-2,1.25){$x_2$};
		\node[novertex] (labelx3) at (-2,-1.25){$x_3$};
		\node[novertex] (labely) at (1.5,0){$y$};

		\draw [densely dotted,thick] (x1) to (x2);
		\draw [snake=coil, segment aspect=2, segment amplitude=1pt,segment length=4pt] (x1) to (x3);
		\draw [snake=zigzag,segment object length=1 pt, segment amplitude=1pt, segment length=3pt] (x2) to (x3);
		\draw [snake=coil, segment aspect=2, segment amplitude=1pt,segment length=4pt] (x1) to (y);
		\draw [densely dotted, thick] (x2) to (y);
		\draw [snake=zigzag,segment object length=1 pt, segment amplitude=1pt, segment length=3pt] (x3) to (y);

		\end{scope}

		\end{tikzpicture}
	\end{center}

	\caption{The two possible cases in Lemma~\ref{lem:claw->no_monochrom_C3} to obtain a colouring in which a fixed triangle is not  monochromatic; the different styles of the edges represent the colours}
	\label{fig:claw->no_monochrom_C3}
\end{figure}
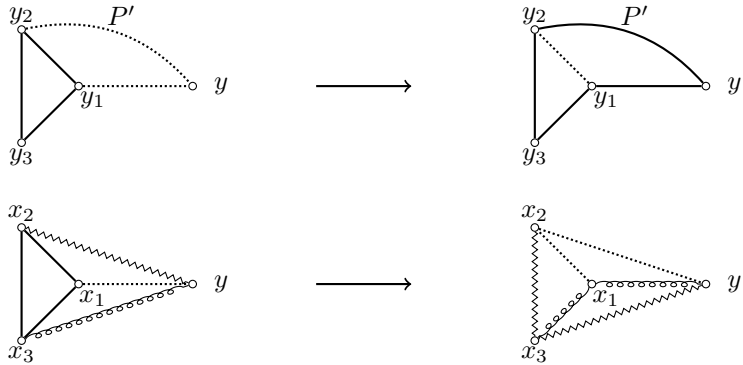

\begin{lemma} \label{lem:claw->no_monochrom_C3}
	Let $H$ be a graph with legal colouring $c$ that contains a clique $\lbrace  x_1, x_2, x_3, y \rbrace$. Then there is a legal colouring $c'$ of $H$ in which the cycle $x_1x_2x_3x_1$ is not monochromatic. 
\end{lemma}

\begin{proof}
	Figure~\ref{fig:claw->no_monochrom_C3} illustrates the recolourings described in this proof. 
	Assume that $x_1x_2x_3x_1$ is monochromatic of colour $i$ in $c$. First assume that
	\begin{equation} \label{eq:condition1}
	\text{an edge of colour $j \coloneqq c(y_1y)$ is adjacent to $y_2$}
	\end{equation}
	for two distinct vertices $y_1,y_2$ in $\lbrace  x_1, x_2, x_3 \rbrace$. Without loss of generality, the path $P'$ of colour $j$ between $y$ and $y_2$ along the path $c^{-1}(j) - \lbrace  yy_1 \rbrace$ does not contain the vertex $y_3$ (where $\lbrace  y_3 \rbrace= \lbrace  x_1,x_2,x_3 \rbrace- \lbrace  y_1, y_2 \rbrace$). Flip the colours of the monochromatic paths $y_1y_2$ and $y_1yP'y_2$, ie set
	$c'(y_1y_2)=j $, $ c'(y_1yP'y_2)=c(y_1y_2)$ and $c'(e)=c(e)$ for all other edges $e \in E(H)$.
	The obtained colouring is legal: By construction, all colour classes are cycles and at most $\lfloor\sfrac{(|V(H)|-1)}{2}\rfloor$ many colours are used. Further, the cycle $x_1x_2x_3x_1$ is not monochromatic.
	\medskip
	
	If~\eqref{eq:condition1} does not hold, we can get rid of one colour. Set  
	$c'(x_1x_2y)=c(x_1y) $, $c'(x_2x_3y)=c(x_2y) $,  $c'(x_3x_1y)=c(x_3y)$, 
	and $c'(e)=c(e)$ for all other edges $e \in E(H)$. By construction, all colour classes are cycles and $x_1x_2x_3x_1$ is not monochromatic.
\end{proof}

Figure~\ref{fig:one_cycle_two_cycles} illustrates the following simple observation.
\begin{observation}\label{obs:one_cycle_two_cycles} 
	Let $P_1$ be an $x_1$-$y_1$-path that is vertex-disjoint from an $x_2$-$y_2$-path $P_2$. 
	Then there are three possibilities to connect $\lbrace  x_1, y_1 \rbrace$ and $ \lbrace x_2,y_2 \rbrace $ by two vertex-disjoint paths that do not intersect $V(P_i)- \lbrace  x_i,y_i \rbrace$ for $i=1, 2$. Two of the possibilities yield a cycle --- the third way leads to two cycles.
\end{observation}

\begin{figure}[bht]
	\begin{center}
		\begin{tikzpicture}[scale = .5]
		\begin{scope}[shift={(-6,0)}]
		\node[novertex] (labelA) at (-2,0){$P_1$};
		\node[novertex] (labelA) at (2,0){$P_2$};
		\node[hvertex] (a1) at (-.5,-1){};
		\node[hvertex] (a2) at (-.5,1){};
		\node[hvertex] (b1) at (.5,-1){};
		\node[hvertex] (b2) at (.5,1){};
		
		\draw [hedge, bend left=90 ] (a1) to (a2);
		\draw [hedge, bend right=90 ] (b1) to (b2);
		\draw [decorate, decoration={snake,amplitude=.4mm,segment length=1mm,post length=0mm}] (a1) -- (b1);
		\draw [decorate, decoration={snake,amplitude=.4mm,segment length=1mm,post length=0mm}] (a2) -- (b2);
		\end{scope}
		
		\begin{scope}[shift={(0,0)}]
		\node[novertex] (labelA) at (-2,0){$P_1$};
		\node[novertex] (labelA) at (2,0){$P_2$};
		\node[hvertex] (a1) at (-.5,-1){};
		\node[hvertex] (a2) at (-.5,1){};
		\node[hvertex] (b1) at (.5,-1){};
		\node[hvertex] (b2) at (.5,1){};
		
		\draw [hedge, bend left=90 ] (a1) to (a2);
		\draw [hedge, bend right=90 ] (b1) to (b2);
		\draw [decorate, decoration={snake,amplitude=.4mm,segment length=1mm,post length=0mm}] (a1) -- (b2);
		\draw [decorate, decoration={snake,amplitude=.4mm,segment length=1mm,post length=0mm}] (a2) -- (b1); (-180:-3:1);
		
		\end{scope}
		\begin{scope}[shift={(6,0)}]
		\node[novertex] (labelA) at (-2,0){$P_1$};
		\node[novertex] (labelA) at (2,0){$P_2$};
		\node[hvertex] (a1) at (-.5,-1){};
		\node[hvertex] (a2) at (-.5,1){};
		\node[hvertex] (b1) at (.5,-1){};
		\node[hvertex] (b2) at (.5,1){};
		
		\draw [hedge, bend left=90 ] (a1) to (a2);
		\draw [hedge, bend right=90 ] (b1) to (b2);
		\draw [decorate, decoration={snake,amplitude=.4mm,segment length=1mm,post length=0mm}] (a1) -- (a2);
		\draw [decorate, decoration={snake,amplitude=.4mm,segment length=1mm,post length=0mm}] (b1) -- (b2);
		
		\end{scope}

		\end{tikzpicture}
	\end{center}

	\caption{The three possible ways to connect the end vertices of two paths $P_1$ and $P_2$; the connection between the end vertices is drawn with jagged lines}
	\label{fig:one_cycle_two_cycles}
\end{figure}
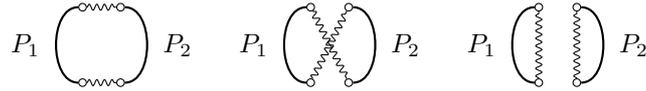


Lemma \ref{lem:adj_5_common_nb}, \ref{lem:adj_4_common_nb} and \ref{lem:nonadj_6_common_nb} are all based on the same elementary fact: 
Let $G$ and $G'$ be graphs with $|V(G)| = |V(G')| + 2$. If $G'$ allows for a cycle decomposition with at most $\lfloor\sfrac{(|V(G')| - 1)}{2}\rfloor$ cycles, then any cycle decomposition of $G$ that uses at most one cycle more than the cycle decomposition of $G'$ shows that $G$ is not a counterexample to Hajós' conjecture.

This fact leads us to the following inductive approach: Given a graph $G$ with two vertices  $u$ and $v$ of degree~$6$, we remove $u$ and $v$ from $G$ and might remove or add edges to obtain a graph $G'$.  If $G'$ has a cycle decomposition with at most $\lfloor\sfrac{(|V(G')| - 1)}{2}\rfloor$ cycles we construct a cycle decomposition of $G$ from it.  We reroute some of the cycles in an appropriate way such that $u$ and $v$ are each touched by two cycles. Now, there remain some edges in $G$ that are not covered.  If those edges form a cycle, we have found a cycle decomposition of $G$. If a cycle is not rerouted to $u$ or $v$ twice, the cycle decomposition of $G$ satisfies Haj\'os' conjecture. 

To describe this inductive approach in a coherent way, we regard the cycle decomposition of $G'$ as a legal colouring. Then we regard the above reroutings as recolourings where we have to make sure that no colour appears twice at $u$ or $v$. If the edges that have not yet received a colour form a cycle, we associate the new colour $\lfloor\sfrac{(|V(G')| - 1)}{2}\rfloor$ to this cycle. The obtained colouring of the edges then uses at most $\lfloor\sfrac{(|V(G')| - 1)}{2}\rfloor$ many colours and each colour class is a cycle. Thus, we have constructed a legal colouring.

\begin{lemma} \label{lem:adj_5_common_nb}
	Let $G$ be an Eulerian graph without legal colouring that contains two adjacent vertices $u$ and $v$ of degree $6$ with common neighbourhood $N= \lbrace x_1, \ldots, x_5 \rbrace$. Define $G'= G - \{u,v\}$ and let $c'$ be a legal colouring of $G'$. 
	\begin{enumerate} [\normalfont\rmfamily(i)]
		\item If $G[N]$ contains a path $P' = y_1y_2y_3y_4$ of length $3$ then $P'$ is monochromatic in $c'$. \label{itm:uv_adj_5common_neighbours_P4}
		\item Let $G[N]$ contain an independent set $S = \{y_1, y_2, y_3\}$ of size $3$.
		If $N$ is not an independent set or if there is a vertex in $G'$ that is adjacent to $y_1$, $y_2$ and $y_3$,  then	$G''=G' + \lbrace  y_1y_2, y_2y_3,y_3y_1 \rbrace$ does not have a legal colouring.\label{itm:uvAdjacent_5common_neighbours_independent_3-set}
		\item If $G[N]$ contains an induced path $y_1y_2y_3y_4$ of length $3$ then $G'' = G' - \lbrace y_2y_3\rbrace + \lbrace y_2y_4,y_4y_1, y_1y_3 \rbrace $ does not have a legal colouring. 
		\label{itm:uvAdjacent_5common_neighbours_inducedP4}
		\item If $G[N]$ contains a triangle $y_1y_2y_3y_1$, a vertex $y_4$ that is not adjacent to $y_1$ and $y_3$ and a vertex $y_5 \in N - \{y_1,y_2,y_3,y_4\}$ adjacent to $y_4$ then $G'' = G'- \lbrace y_1y_3 \rbrace + \lbrace y_1y_4, y_3y_4 \rbrace$ does not have a legal colouring.
		\label{itm:uvAdjacent_5common_neighbours_C3+2adjacent_nonedges}
	\end{enumerate}
\end{lemma}

\begin{proof}[Proof of~\eqref{itm:uv_adj_5common_neighbours_P4}]
	If $y_3y_4$ has a colour different from $y_1y_2$ and $y_2y_3$, then set
	\begin{gather*}
	c(y_1uy_2)=c'(y_1y_2) \qquad  c(y_2vy_3)=c'(y_2y_3) \qquad  c(y_3uvy_4)=c'(y_3y_4).
	\end{gather*}
	If $y_2y_3$ has a colour different from $y_1y_2$ and $y_3y_4$, then set
	\begin{gather*}
	c(y_1uy_2)=c'(y_1y_2) \qquad  c(y_2vuy_3)=c'(y_2y_3) \qquad  c(y_3vy_4)=c'(y_3y_4).
	\end{gather*}
	The case distinction makes sure that  the modified colour classes remain cycles. 
	By further setting
	$
	c(y_1y_2y_3y_4uy_5vy_1)=\ \left\lfloor\sfrac{(|V(G)|-1)}{2}\right\rfloor 
	$
	and $c(e)=c'(e) $ for all other edges $ e$ we have constructed a legal colouring $c$ of $G$. 
\end{proof} 

\begin{proof}[Proof of~\eqref{itm:uvAdjacent_5common_neighbours_independent_3-set}]		
	Set $\{y_4,y_5\} = N - \{y_1, y_2, y_3\}$ and let $c''$ be a legal colouring of $G''$. \\
	First assume  that $c''(y_1y_2) \notin \{c''(y_2y_3), c''(y_3y_1)\}$.  Then one can easily check that the following is a legal colouring of $G$. 
	\begin{gather}
	c(y_2uvy_1)=c''(y_2y_1) \quad  c(y_2vy_3)=c''(y_2y_3) \quad  c(y_3uy_1)=c''(y_3y_1) \nonumber\\
	c(y_4uy_5vy_4)= \left\lfloor\sfrac{(|V(G)|-1)}{2}\right\rfloor \nonumber\\
	c(e)=c''(e) \text{ for all other edges } e  \label{eq:recol_indep_3set_A}
	\end{gather} 
	By symmetry, we are done unless the triangle  $y_1y_2y_3y_1$ is monochromatic in $c''$.  By Lemma~\ref{lem:claw->no_monochrom_C3}, we can suppose that there is no vertex $y$ in $G''$ that is adjacent to $y_1$, $y_2$ and $y_3$. Suppose that $N$ is not independent. Without loss of generality, we can assume that $G[N]$ contains an edge, say $y_4y_1$ incident to one of the vertices of the independent $3$-set. (Otherwise, we can choose another suitable independent $3$-set in $G[N]$).
	Then  by construction the following is a legal colouring of $G$. 
	\begin{gather*}
	c(y_1uvy_4)=c''(y_1y_4) \quad  c(y_2uy_3)=c''(y_2y_3) \quad  c(y_2vy_3)=c''(y_2y_1y_3) \nonumber \\
	c(y_1y_4uy_5vy_1)= \left\lfloor\sfrac{(|V(G)|-1)}{2}\right\rfloor  \nonumber\\
	c(e)=c''(e) \text{ for all other edges $e$} \label{eq:recol_indep_3set_B}\qedhere
	\end{gather*}
\end{proof}

\begin{proof}[Proof of~\eqref{itm:uvAdjacent_5common_neighbours_inducedP4}]
	Let $G''$ have a legal colouring $c''$ and let $y_5$ be the unique vertex in $ N- \{y_1, y_2, y_3, y_4\}$. \\
	If $c''(y_2y_4)=c''(y_4y_1)=c''(y_1y_3)$, set
	\begin{gather*}
	c(y_1uy_2)=c''(y_1y_2) \qquad c(y_2vuy_3)=c''(y_2y_4y_1y_3)  \qquad c(y_3vy_4)=c''(y_3y_4)\\
	c(uy_5vy_1y_2y_3y_4u)= \left\lfloor \sfrac{(|V(G)|-1)}{2} \right\rfloor.
	\end{gather*}
	If $c''(y_1y_3)$ is different from $c''(y_2y_4)$ and $c''(y_4y_1)$, set
	\begin{gather*}
	c(y_1uvy_3)=c''(y_1y_3)  \qquad c(y_4vy_1)=c''(y_4y_1) \qquad c(y_2uy_4)=c''(y_2y_4)  \\
	c(uy_5vy_2y_3u)= \left\lfloor \sfrac{(|V(G)|-1)}{2} \right\rfloor.
	\end{gather*}
	If $c''(y_2y_4)$ is different from $c''(y_1y_3)$ and $c''(y_1y_4)$, the colouring is defined si\-mi\-larly by relabelling the vertices $y_1, \dots, y_5$. 	\\
	If $c''(y_4y_1)$ is different from $c''(y_1y_3)$ and $c''(y_2y_4)$, set
	\begin{gather*}
	c(y_1vy_3)=c''(y_1y_3)  \qquad c(y_4vuy_1)=c''(y_4y_1) \qquad c(y_2uy_4)=c''(y_2y_4)  \\
	c(uy_5vy_2y_3u)= \left\lfloor \sfrac{(|V(G)|-1)}{2} \right\rfloor.
	\end{gather*}
	
	Further set $c(e)=c''(e)$  for all other edges $e$ in all cases. 
	Again, the case distinction makes sure that all colour classes are cycles and we have constructed a legal colouring. 
\end{proof}

\begin{proof}[Proof of \eqref{itm:uvAdjacent_5common_neighbours_C3+2adjacent_nonedges}]
	Let $c''$ be a legal colouring of $G''$.
	First assume that $c''(y_2y_3) \notin \lbrace  c''(y_3y_4), c''(y_1y_4) \rbrace$.
	Then set
	\begin{gather*}
	c(y_2vuy_3)=c''(y_2y_3) \qquad c(y_1uy_4)=c''(y_1y_4) \qquad c(y_3vy_4)=c''(y_3y_4) \\
	c(uy_5vy_1y_3y_2u)= \left\lfloor \sfrac{(|V(G)|-1)}{2} \right\rfloor.
	\end{gather*}
	If $c''(y_1y_2) \notin \lbrace  c''(y_3y_4), c''(y_1y_4) \rbrace$, the colouring is defined as above by interchanging the roles of $y_1$ and $y_3$. 
	
	Now assume that $c''(y_2y_3) , c''(y_1y_2) \in \lbrace  c''(y_3y_4), c''(y_1y_4) \rbrace$. 
	If $  c''(y_3y_4)= c''(y_1y_4)$, then  the cycle $y_1y_2y_3y_4y_1$ is monochromatic. Set
	\begin{gather*}
	c(y_4uvy_5)=c''(y_4y_5) \\
	c(y_1vy_3y_2uy_1)=c''(y_1y_2y_3y_4y_1)\qquad c(y_1y_3uy_5y_4vy_2y_1)= \left\lfloor \sfrac{(|V(G)|-1)}{2} \right\rfloor.
	\end{gather*}
	If $ c''(y_3y_4) \not= c''(y_1y_4)$, then either $c''(y_2y_3)= c''(y_3y_4)$ or $c''(y_2y_3)= c''(y_1y_4)$. 
	If $c''(y_2y_3)= c''(y_3y_4)$, set
	\begin{gather*}
	c(y_2uy_3vy_4)=c''(y_2y_3y_4) \qquad c(y_1vuy_4)=c''(y_1y_4) \\
	c(y_1uy_5vy_2y_3y_1)= \left\lfloor \sfrac{(|V(G)|-1)}{2} \right\rfloor.
	\end{gather*}
	If $c''(y_2y_3)= c''(y_1y_4)$, set
	\begin{gather*}
	c(y_2uy_3)=c''(y_2y_3) \qquad c(y_1vy_4)=c''(y_1y_4)  \qquad c(y_3vuy_4)=c''(y_3y_4)\\
	c(y_1uy_5vy_2y_3y_1)= \left\lfloor \sfrac{(|V(G)|-1)}{2} \right\rfloor.
	\end{gather*}
	By setting $c(e)=c''(e)$  for all other edges $e$ we have constructed a legal colouring for $G$ in all cases. 
\end{proof}

If $u$ and $v$ are adjacent degree-$6$ vertices that  have a common neighbourhood~$N$ of size $4$, we call the two vertices that are adjacent with exactly one of $u, v$ the \emph{private neighbours} of $u$ and $v$. Here, we denote them by $x_u$ and $x_v$.  If there is a $x_u$-$x_v$-path $P$ in $G - \lbrace u, v \rbrace -N$, it is possible to translate all techniques of Lemma~\ref{lem:nonadj_6_common_nb}.  It suffices to delete $u$, $v$ and $E(P)$ to obtain another Eulerian graph: In all recolourings of Lemma~\ref{lem:nonadj_6_common_nb}, the edges $uy,vy$ for one vertex $y \in N$ were contained in the new colour class $\left\lfloor \sfrac{(|V(G)|-1)}{2} \right\rfloor $. If we have two private neighbours $x_u$ and $x_v$ it suffices to replace the path $uyv$ by the path $u x_uPx_v v$ in this colour class. This means, we can regard $x_uPx_v$ as a single vertex $y$.

\begin{lemma} \label{lem:adj_4_common_nb}
	Let $G$ be an Eulerian graph without legal colouring that contains two adjacent vertices $u$ and $v$ of degree $6$ with common neighbourhood $N= \lbrace x_1, \ldots, x_4 \rbrace$ and $N_G(u) = N \cup  \{x_u, v\}$ as well as $N_G(v) = N \cup \{x_v, u\}$.  Let $P$ be an $x_u$-$x_v$-path in $G - \{u,v\}-N $. Define $G'= G - \{u,v\}-E(P)$ and let $c'$ be a legal colouring of $G'$.  
	\begin{enumerate} [\normalfont\rmfamily(i)]
		\item If $G[N \cup \lbrace x_u,x_v \rbrace]$ contains a path $P'=y_1y_2y_3y_4$ with $y_2,y_3,y_4 \in N$ of length $3$ then $P'$ is monochromatic in $c'$.  \label{itm:uvAdjacent_4common_neighbours_P4}
		\item 
		Let $G[N]$ contain an independent set $S = \{y_1, y_2, y_3\}$ of size $3$.
		If $G[N\cup \lbrace  x_u,x_v \rbrace]$ contains an edge $x_ix_j \not= x_ux_v$ or if there is a vertex in $G'$ that is adjacent to $y_1$, $y_2$ and $y_3$  then $G''=G' + \lbrace  y_1y_2, y_2y_3,y_3y_1 \rbrace$ does not have a legal colouring. 
		\label{itm:uvAdjacent_4common_neighbours_independent_3-set}
		\item 	If $G[N \cup \lbrace  x_u,x_v \rbrace]$ does not contain the edges $x_uy_1, y_1y_2,y_2x_v $ for two vertices $y_1,y_2 \in N$ but contains an edge with end vertex $y_1$ or $y_2$ then $G''=G-\{u,v\}+\{x_uy_1,y_1y_2,y_2x_v\}$ does not have a legal colouring.
		\label{itm:uvAdjacent_4common_neighbours_notP4+edge}
		\item If $G$ contains the edges $y_1y_2,y_3y_4, y_1y_5$ with $y_1,y_2,y_3,y_4\in N$ and $y_5\in\{x_u,x_v\}$ but not the edges $y_1y_3,y_2y_3$ then $G''=G'-\{y_1y_2\}+\{y_1y_3,y_3y_2\}$ does not have a legal colouring. 
		\label{itm:uvAdjacent_4common_neighbours_some_edges_there_some_not}
		\item If $G[N\cup \{x_u,x_v\}]$ contains a triangle $y_1y_2y_3y_1$ with $y_1,y_2,y_3\in N$, a vertex $y_4\in N-\{y_1,y_2,y_3\}$ that is not adjacent to $y_1$ and $y_3$ and a vertex $y_5\in \{x_u,x_v\}$ adjacent to $y_4$ then $G'' = G'- \lbrace y_1y_3 \rbrace + \lbrace y_1y_4, y_3y_4 \rbrace$ does not have a legal colouring.	\label{itm:uvAdjacent_4common_neighbours_C3+2adjacent_nonedges}
	\end{enumerate}
\end{lemma}

\begin{proof}[Proof of~\eqref{itm:uvAdjacent_4common_neighbours_P4}]
	The proof is very similar to the proof of Lemma~\ref{lem:adj_5_common_nb}.\eqref{itm:uvAdjacent_4common_neighbours_P4} if we regard $x_uPx_v$ as one single vertex. We will nevertheless give a detailed proof. By symmetry of $u$ and $v$ (and thus of $x_u$ and $x_v$), we can assume that $y_1$ is either contained in $N$ or is equal to $x_u$.
	Suppose that $P$ is not monochromatic.

	If $c'(y_1y_2) \notin \{c'(y_2y_3), c'(y_3y_4)\}$, then set 
	$$c(y_1uvy_2)=c'(y_1y_2) \quad  c(y_2uy_3)=c'(y_2y_3) \quad c(y_3vy_4)=c'(y_3y_4).$$ 
	If $c'(y_2y_3) \notin  \{c'(y_1y_2), c'(y_3y_4) \}$, then set 
	$$c(y_1uy_2)=c'(y_1y_2) \quad   c(y_2vuy_3)=c'(y_2y_3) \quad c(y_3vy_4)=c'(y_3y_4).$$
	If $c'(y_3y_4) \notin \{c'(y_1y_2), c'(y_2y_3)\}$, then set 
	$$c(y_1uy_2)=c'(y_1y_2) \quad  c(y_2vy_3)=c'(y_2y_3) \quad c(y_3uvy_4)=c'(y_3y_4).$$ 
	
	If $y_1 \in N$ the following completes  by construction a legal colouring $c$ of~$G$:
	\begin{gather*}
	c(y_1y_2y_3y_4ux_uPx_vvy_1)= \left\lfloor\sfrac{(|V(G)|-1)}{2}\right\rfloor \\
	c(e)=c'(e) \text{ for all other edges } e
	\end{gather*}
	
	Now suppose that $y_1=x_u$ and that $x_4,x_v$ are not contained in the path $y_1y_2y_3y_4$.
	Then, the following completes  by construction a legal colouring $c$ of~$G$:
	\begin{gather*}
	c(y_1y_2y_3y_4ux_4vx_vPy_1)=\ \left\lfloor\sfrac{(|V(G)|-1)}{2}\right\rfloor\\
	c(e)=c'(e) \text{ for all other edges } e \qedhere
	\end{gather*}
\end{proof}

\begin{proof}[Proof of~\eqref{itm:uvAdjacent_4common_neighbours_independent_3-set}]
	The proof is very similar to the proof of Lemma~\ref{lem:adj_5_common_nb}.\eqref{itm:uvAdjacent_4common_neighbours_independent_3-set} if we regard $x_uPx_v$ as one single vertex. 
\end{proof}

\begin{proof}[Proof of~\eqref{itm:uvAdjacent_4common_neighbours_notP4+edge}]
	Assume that $c''$ is  a legal colouring of $G''$ and let $\{y_3,y_4\}=N-\{y_1,y_2\}$. By symmetry of $u$ and $v$ (and thus of $y_1$ and $y_2$) we can suppose that $y_1y_4\in E(G)$. \\
	If $y_1x_u$ has a colour different from the colour of $y_1y_2$ and $y_2x_v$, set
	\begin{gather*}
	c(x_uuvy_1)=c''(x_uy_1) \quad  c(y_1uy_2)=c''(y_1y_2) \quad c(y_2vx_v)=c''(y_2x_v) \\
	c(y_3uy_4vy_3)= \left\lfloor\sfrac{(|V(G)|-1)}{2}\right\rfloor.
	\end{gather*}
	An analogous colouring can be defined if $x_vy_2$ has a colour different from the colour of $y_1y_2$ and $x_uy_1$.\\
	If $y_1y_2$ has a colour different from the colour of $y_1x_u$ and $y_2x_v$, then set
	\begin{gather*}
	c(x_uuy_1)=c''(x_uy_1) \quad  c(y_1vuy_2)=c''(y_1y_2) \quad c(y_2vx_v)=c''(y_2x_v) \\
	c(y_3uy_4vy_3)= \left\lfloor\sfrac{(|V(G)|-1)}{2}\right\rfloor.
	\end{gather*}
	Now suppose that all three edges $x_uy_1, y_1y_2,y_2x_v$ have the same colour. Then, $y_1y_4$
	has a different colour. Set
	\begin{gather*}
	c(x_uuy_2)=c''(x_uy_1y_2) \, \,  c(y_1uvy_4)=c''(y_1y_4) \, \,  c(y_2vx_v)=c''(y_2x_v) \\
	c(uy_3vy_1y_4u)= \left\lfloor\sfrac{(|V(G)|-1)}{2}\right\rfloor.
	\end{gather*}
	In all cases, set $c(e)=c''(e)$  for all other edges $e$. The case distinction now makes sure that we constructed a legal colouring for $G$.
\end{proof}

\begin{proof}[Proof of~\eqref{itm:uvAdjacent_4common_neighbours_some_edges_there_some_not}]
	Assume that $c''$ is a legal colouring of $G''$. Without loss of gene\-rality let $y_5=x_u$.
	
	First  suppose that all three edges $x_uy_1, y_1y_3,y_3y_2$ have the same colour. Then, $y_3y_4$
	has a different colour and the following gives by construction a legal colouring for $G$:
	\begin{gather*}
	c(x_uuy_1)=c''(x_uy_1) \, \,  c(y_1vy_2)=c''(y_1y_3y_2) \, \,  c(y_3uvy_4)=c''(y_3y_4) \\
	c(uy_2y_1x_uPx_vvy_3y_4u)= \left\lfloor\sfrac{(|V(G)|-1)}{2}\right\rfloor\\
	c(e)=c'(e) \text{ for all other edges } e
	\end{gather*}
	Now suppose that $x_uy_1, y_1y_3,y_3y_2$ is not monochromatic. 
	\\
	If $x_uy_1$ has a colour different from the colours of $y_1y_3$ and $y_3y_2$, set
	$$c(x_uuvy_1)=c''(x_uy_1) \quad  c(y_1uy_3)=c''(y_1y_3) \quad c(y_3vy_2)=c''(y_3y_2).$$
	If $y_1y_3$ has a colour different from the colours of $x_uy_1$ and $y_3y_2$, set $$c(x_uuy_1)=c''(x_uy_1) \quad  c(y_1vuy_3)=c''(y_1y_3) \quad c(y_3vy_2)=c''(y_3y_2).$$
	If $y_3y_4$ has a colour different from the colours of $x_uy_1$ and $y_1y_3$, set $$c(x_uuy_1)=c''(x_uy_1) \quad  c(y_1vy_3)=c''(y_1y_3) \quad c(y_3uvy_2)=c''(y_3y_2).$$  
	By setting $c(uy_2y_1x_uPx_vvy_4u)= \left\lfloor\sfrac{(|V(G)|-1)}{2}\right\rfloor$ and $c(e)=c''(e)$  for all other edges $e$, we obtain by construction in all cases a legal colouring for $G$.
\end{proof}

\begin{proof}[Proof of~\eqref{itm:uvAdjacent_4common_neighbours_C3+2adjacent_nonedges}]
	The proof is very similar to the proof of Lemma~\ref{lem:adj_5_common_nb}.\eqref{itm:uvAdjacent_5common_neighbours_C3+2adjacent_nonedges} if we regard $x_uPx_v$ as one single vertex. 
\end{proof}

In our last recolouring lemma, we consider two degree-$6$ vertices that are not adjacent but have six common neighbours $x_1, \ldots, x_6$.  Some of the recolouring techniques of this lemma need a somewhat deeper look into the cycle\- decomposition. They rely on a generalisation of the recolourings used in Lemma~\ref{lem:adj_5_common_nb} and~\ref{lem:claw->no_monochrom_C3}. 
We introduce two pieces of notation. 
For two distinct vertices $x_i, x_j \in N=\lbrace  x_1, \ldots, x_6 \rbrace $, a path $P_{x_ix_j}$  always denotes an $x_i$-$x_j$-path that is not intersecting with $N- \{x_i, x_j\}$. 

For a cycle $C$ and two distinct vertices $x_i, x_j \in N=\lbrace  x_1, \ldots, x_6 \rbrace \cap V(C)$ there are two $x_i-x_j$-paths along $C$. If there is a unique path that is not intersecting with $N- \{x_i, x_j\}$, we denote this path by $C_{x_ix_j}$.

\begin{lemma} \label{lem:nonadj_6_common_nb} 
	Let $G$ be an Eulerian graph without legal colouring and let $G$ contain two degree-$6$ vertices $u$ and $v$ with common neighbourhood $N=\lbrace x_1, \ldots, x_6 \rbrace$. 
	Define $G'= G- \lbrace  u,v \rbrace$ and let $c'$ be a legal colouring of $G'$. 
	\begin{enumerate} [\normalfont\rmfamily(i)]
		\item If $G'$ contains two vertex-disjoint paths $P_{y_1y_2} P_{y_2y_3}$ and $P_{y_1'y_2'} P_{y_2'y_3'}$ with $\{y_1,y_2,y_3,y_1',y_2',y_3'\}=N$ 
		where the four paths $P_{y_1y_2}, P_{y_2y_3}, P_{y_1'y_2'}, P_{y_2'y_3'}$  are monochromatic in~$c'$,
		then at least three of the four paths have the same colour in~$c'$. 
		\label{itm:uv_not_adj_6common_neighbours_P3_P3}
		\item Let $G'$ contain a path $P'=P_{y_1y_2} P_{y_2y_3} P_{y_3y_4}P_{y_4y_5}$ with $\{y_1,\ldots,y_5\}\subset N$ where $P_{y_iy_{i+1}}$ is monochromatic in~$c'$ for each $i\in \{1,2,3,4\}$. Then $c'(P_{y_1y_2})=c'(P_{y_3y_4})$ or $c'(P_{y_2y_3})=c'(P_{y_4y_5})$. 
		\label{itm:uv_not_adj_6common_neighbours_P5} 
		\item  \label{itm:uv_not_adj_6common_neighbours_indep_3set} 
		If  $G[N]$ contains an independent set $S=\{y_1, y_2, y_3\}$ of size $3$ and if 
		$G[N]$ contains at least one edge or there is a vertex in $G'$ that is adjacent to $y_1$, $y_2$ and $y_3$ then 
		$G'' = G' + \{y_1y_2, y_2y_3, y_3y_1\}$ does not have a legal colouring. 
		\item If $G[N]$ contains a path $P'=y_1y_2y_3y_4$ of length $3$, then $P'$ is monochromatic in $c'$.
		\label{itm:uv_not_adj_6common_neighbours_P4} 
	\end{enumerate}
\end{lemma}

\begin{proof}[Proof of~\eqref{itm:uv_not_adj_6common_neighbours_P3_P3}]
	Suppose that less than three of the paths have the same colour. Then, without loss of generality $c'(P_{y_1'y_2'}) \not=c'(P_{y_1y_2})$ and  $c'( P_{y_2'y_3'}) \not=c'(P_{y_2y_3})$ and the following is by construction a legal colouring of $G$:
	\begin{gather*}
	c(y_1uy_2)=c'(P_{y_1y_2})  \qquad  c(y_2vy_3)=c'(P_{y_2y_3})  \\ 
	c(y_1'uy_2')=c'(P_{y_1'y_2'})  \qquad  c(y_2'vy_3')=c'(P_{y_2'y_3'})\\
	c(y_1P_{y_1y_2}P_{y_2y_3}y_3uy_3'P_{y_3'y_2'}P_{y_2'y_1'}y_1'vy_1)=\left\lfloor\sfrac{(|V(G)|-1)}{2}\right\rfloor  \\
	c(e)=c'(e) \text{ for all other edges } e \qedhere
	\end{gather*}
\end{proof}

\begin{proof}[Proof of~\eqref{itm:uv_not_adj_6common_neighbours_P5}]
	Suppose that $c'( P_{y_1y_2}) \not=c'( P_{y_3y_4})$ and $c'( P_{y_2y_3})\not=c'( P_{y_4y_5})$ and let $y_6$ be the vertex of $N$ not contained in $P'$. Then, the following is by construction a legal colouring of $G$:
	\begin{gather*}
	c(y_1uy_2)=c'(P_{y_1y_2}) \qquad  c(y_2vy_3)=c'(P_{y_2y_3})  \\
	c(y_3uy_4)=c'(P_{y_3y_4}) \qquad  c(y_4vy_5)=c'(P_{y_4y_5}) \\
	c( y_1P_{y_1y_2} P_{y_2y_3} P_{y_3y_4} P_{y_4y_5}y_5uy_6vy_1)=\left\lfloor\sfrac{(|V(G)|-1)}{2}\right\rfloor  \\
	c(e)=c'(e) \text{ for all other edges } e \qedhere
	\end{gather*}
\end{proof}

\begin{proof}[Proof of~\eqref{itm:uv_not_adj_6common_neighbours_indep_3set}]
	The proof uses ideas of the proof of Lemma~\ref{lem:adj_5_common_nb}.\eqref{itm:uvAdjacent_5common_neighbours_independent_3-set}.
	
	Let $c''$ be a legal colouring of $G''$. First suppose  that $i \coloneqq c''(y_1y_2) \notin \{c''(y_2y_3), c''(y_3y_1)\}$ and let $C=c^{-1}(i)$ be the monochromatic cycle in $G''$ with colour $i$.
	
	If there is a vertex $y_4 \in N - \{y_1, y_2, y_3\}$ that is not contained in $C$ 
	set $\{y_5, y_6\} = N- \{y_1, y_2, y_3, y_4\}$ and use the recolouring~\eqref{eq:recol_indep_3set_A} where the edge $uv$ is replaced by the path $uy_4v$. 
	
	Otherwise, $\{y_4, y_5, y_6\} \coloneqq N- \{y_1, y_2, y_3\}$ is a subset of $V(C)$. Then
	without loss of generality  $C_{y_3y_4}$ and $C_{y_4y_5}$ exist. By construction, the following is  a legal colouring of $G$:
	\begin{gather*}
	c(y_2uy_3) = c''(y_2y_3) 
	\quad c(y_1vy_2) = c''(y_1y_2) 
	\quad c(y_1uy_4vy_5) = c''(y_3y_1)\\
	c(y_3C_{y_3y_4} C_{y_4y_5}y_5uy_6vy_3) = \left\lfloor\sfrac{(|V(G)|-1)}{2}\right\rfloor\\
	\text{$c(e)=c''(e)$ for all other edges $e$.}	
	\end{gather*}
	
	\medskip
	Assume that the triangle  $y_1y_2y_3y_1$ is monochromatic in $c''$.  By Lemma~\ref{lem:claw->no_monochrom_C3},  there is no vertex $y$ in $G'$ that is adjacent to $y_1$, $y_2$ and $y_3$. Suppose that $N$ is not independent. Without loss of generality  $G[N]$ contains the edge $y_4y_1$. 
	Set $ \lbrace  y_5,y_6 \rbrace= N- \lbrace  y_1, y_2, y_3, y_4 \rbrace$. 
	
	If there is a vertex in  $\lbrace  y_5,y_6 \rbrace$, say $y_6$, that is not contained in the cycle $C=c^{-1}(j)$ of colour $j \coloneqq c'(y_1y_4)$,  use the recolouring~\eqref{eq:recol_indep_3set_A} where the edge $uv$ is replaced by the path $uy_6v$.

	If $y_5$ and $y_6$ are both contained in $C$, let $S$ be the segment of $C- \lbrace y_1y_4 \rbrace$ that connects $y_4$ with $y_5$. By symmetry of $y_5$ and $y_6$, we can suppose that $y_6 \notin S$. 
	By construction, the following is a legal colouring of $G$:
	\begin{gather*}
	c(y_1vy_4)=c''(y_1y_4) \quad  c(y_5uy_4)=c''(S)  \\
	c(y_2uy_3vy_2)= c(y_1y_2y_3y_1) \quad
	c(y_1y_4Sy_5vy_6uy_1)= \left\lfloor\sfrac{(|V(G)|-1)}{2}\right\rfloor\\
	\text{$c(e)=c''(e)$ for all other edges $e$} \qedhere
	\end{gather*}
\end{proof}

\begin{proof}[Proof of~\eqref{itm:uv_not_adj_6common_neighbours_P4}]
	Suppose that $P$ is not monochromatic in $c'$ and set $\{y_5,y_6\} = N- \{y_1, y_2, y_3, y_4\}$. 
	
	\medskip
	
	First assume that $c'(y_3y_4)  \notin \{c'(y_1y_2), c'(y_2y_3)\}$.
	Let $C$ be the cycle of colour $c'(y_3y_4)$ in $G'$. \\
	If there is a vertex in $\{y_5, y_6\}$, say $y_5$, that is not in $C$, then by construction the following is a legal colouring of $G$:
	\begin{gather*}
	c(y_1uy_2) = c'(y_1y_2) \quad c(y_2vy_3) = c'(y_2y_3) \quad c(y_3uy_5vy_4) = c'(y_3y_4)\\
	c(y_1y_2y_3y_4uy_6vy_1) = \left\lfloor\sfrac{(|V(G)|-1)}{2}\right\rfloor\\
	c(e)=c'(e) \text{ for all other edges } e
	\end{gather*}
	Now assume that $y_5$ and $y_6$ are contained in $C$. If $C_{y_5y_1}$, $C_{y_6y_1}$, $C_{y_5y_4}$ or $C_{y_6y_4}$ exists then we can apply~\eqref{itm:uv_not_adj_6common_neighbours_P5}. Thus, $C_{y_5y_6}$ must exist and by symmetry $C_{y_5y_2}$ and $C_{y_6y_3}$ exist.
	We can apply~\eqref{itm:uv_not_adj_6common_neighbours_P5} to $y_1y_2, y_2y_3, C_{y_3y_6}$ and $C_{y_5y_6}$. 
	
	\medskip
	
	Thus, for the rest of the proof we can assume that 
	$$c'(y_2y_3) \eqqcolon i \notin \{c'(y_1y_2), c'(y_3y_4)\}.$$
	Let $C'=c^{-1}(i)$ be the cycle of colour $i$ in $G'$. If there is a vertex in $\{y_5, y_6\}$, say $y_5$, that is not in $C'$, then by construction the following is a legal colouring of $G$:
	\begin{gather*}
	c(y_1uy_2) = c'(y_1y_2) \quad c(y_3vy_4) = c'(y_3y_4) \quad c(y_2vy_5uy_3) = c'(y_2y_3)\\
	c(y_1y_2y_3y_4uy_6vy_1) = \left\lfloor\sfrac{(|V(G)|-1)}{2}\right\rfloor\\
	c(e)=c'(e) \text{ for all other edges } e
	\end{gather*}
	Thus, we can assume that 
	\begin{equation*}
	y_5 \text{ and } y_6 \text{ are contained in } C'.
	\end{equation*}
	Now, there are three cases up to symmetry: $y_1$ and $y_4$ both are not contained in~$C'$, $y_1$ is contained in $C'$ but $y_4$ is not, and $y_1$ and $y_4$ are both contained in $C'$.
	
	\medskip
	
	First assume that $y_1$ and $y_4$ are not contained in $C'$. Then, by symmetry, $C'$ is the cycle consisting of $y_2y_3, C'_{y_3y_6}, C'_{y_6y_5}, C'_{y_5y_2}$. We are done by applying~\eqref{itm:uv_not_adj_6common_neighbours_P3_P3} to the vertex-disjoint paths $y_1y_2, C'_{y_2y_5}$ and $y_4y_3, C'_{y_3y_6}$. 
	
	\medskip
	
	Next assume that $y_1$ is contained in $C'$ and $y_4$ is not contained in $C'$. 
	First suppose that $C'_{y_6y_3}$ exists. As $C'_{y_5y_1}$ or $C'_{y_5y_2}$ must exist, we are done with~\eqref{itm:uv_not_adj_6common_neighbours_P3_P3}.  
	
	By symmetry, we can now suppose that neither $C'_{y_6y_3}$ nor $C'_{y_5y_3}$ exists. Then $C'_{y_3y_1}$ exists. We can suppose without loss of generality that 
	$C'$ is the cycle consisting of $C'_{y_1y_3}, y_3y_2, C'_{y_2y_6}, C'_{y_6y_5}, C'_{y_5y_1}$ and 
	by construction the following is a legal colouring of $G$:
	\begin{gather*}
	c(y_1uy_2) = c'(y_1y_2) \quad c(y_3vy_4) = c'(y_3y_4)\\
	c(y_3y_2vy_1C'_{y_1y_5}C'_{y_5y_6}y_6uy_4y_3) = i\\ 
	c(y_3uy_5vy_6C'_{y_6y_2}y_2y_1C'_{y_1y_3}y_3) 
	= \left\lfloor\sfrac{(|V(G)|-1)}{2}\right\rfloor\\
	c(e)=c'(e) \text{ for all other edges } e
	\end{gather*}
	Last, assume that $y_1$ and $y_4$ are both contained in $C'$.
	First suppose that $C'_{y_5y_6}$ does not exist. Without loss of generality, we can suppose that $C'_{y_5y_1}$ exists.  Now neither $C'_{y_6y_3}$ nor $C'_{y_6y_4}$ exists; otherwise we are done with~\eqref{itm:uv_not_adj_6common_neighbours_P3_P3}. Thus, $C'_{y_6y_1}$ and $C'_{y_6y_2}$ must exist. Thus, $C'_{y_5y_4}$ exists and we are done with~\eqref{itm:uv_not_adj_6common_neighbours_P3_P3}.

	Now suppose that $C'_{y_5y_6}$ exists. First suppose that $C'_{y_5y_2}$ exists. Then, we are done with~\eqref{itm:uv_not_adj_6common_neighbours_P3_P3} if $C'_{y_6y_3}$ or $C'_{y_6y_4}$ exists. 
	As $C'$ is a cycle,  $C'_{y_6y_1}$ and thus also $C'_{y_4y_1}$ and $C'_{y_4y_3}$ exist. The following is by construction a legal colouring of $G$:  
	\begin{gather*}
	c(y_3uy_4) = c'(y_3y_4) \quad c(y_1vy_2) = c'(y_1y_2)\\
	c(y_5vy_3C'_{y_3y_4}C'_{y_4y_1}y_1y_2uy_6C'_{y_6y_5}y_5) = i\\
	c(y_2y_3y_4vy_6C'_{y_6y_1}y_1uy_5C'_{y_5y_2}y_2) = \left\lfloor\sfrac{(|V(G)|-1)}{2}\right\rfloor\\
	c(e)=c'(e) \text{ for all other edges } e
	\end{gather*}
	Thus, we can suppose that none of $C'_{y_5y_2}$, $C'_{y_5y_3}$, $C'_{y_6y_2}$, $C'_{y_6y_3}$ exists. Without loss of generality,  $C'_{y_5y_1}$ exists. As $C'_{y_6y_4}$ must exist we are done with~\eqref{itm:uv_not_adj_6common_neighbours_P3_P3}. 
\end{proof}

\section{Proofs for the reducible structures} 
\label{sec:proofs_thms}

In this section we prove Lemma~\ref{lem:uv_adj_5common_nb_forbidden_structures},
\ref{lem:uv_adj_4common_nb_forbidden_structures},
\ref{lem:nonadj_6_common_nb_forbidden}, \ref{lem:deg2_or_4_neighbourhood_odd} and 
\ref{lem:deg6_neighbourhood}.
In the first three proofs, we use the following observation:

\begin{observation} \label{obs:no_deg3} 
Let $x$ be a vertex of degree at least $3$ in a graph $H$ with a legal colouring.  
Then the neighbourhood $N_x$ of $x$ contains an independent set of size $3$ or $G \left[\lbrace x \rbrace \cup N_x \right]$ contains a path of length $3$ that is not monochromatic. 
\end{observation}

\begin{proof} [Proof of Lemma~\ref{lem:uv_adj_5common_nb_forbidden_structures}] 
If $G[N]$ contains a vertex of degree at least $3$ we are done  by applying Observation~\ref{obs:no_deg3}
as well as  Lemma~\ref{lem:adj_5_common_nb}.\eqref{itm:uv_adj_5common_neighbours_P4} and \eqref{itm:uvAdjacent_5common_neighbours_independent_3-set}.

Now, suppose that $G[N]$ contains a vertex, say $x_1$ of degree $0$.
	As we have seen, $G[N]$ contains no vertex of degree $3$ or $4$. Thus, $G[N]-\lbrace x_1 \rbrace$ contains two non-adjacent vertices, say $x_2$ and $x_3$. Then, $\lbrace x_1,x_2,x_3 \rbrace$ is an independent set and we are done by Lemma~\ref{lem:adj_5_common_nb}.\eqref{itm:uvAdjacent_5common_neighbours_independent_3-set}.
	
	We can conclude that all vertices in $G[N]$ have degree $1$ or $2$. Consequently, the graph is isomorphic to $C_5$, $K_3 \dot{\cup} P_2$, $P_3 \dot{ \cup} P_2$ or $P_5$. The $5$-cycle $C_5$ contains an induced $P_4$, the graph $K_3 \dot{\cup} P_2$ contains a triangle and a vertex that is not adjacent to two of the triangle vertices, the latter two graphs contain an independent set of size $3$.
	Thus, we are done by~\eqref{itm:uvAdjacent_5common_neighbours_inducedP4}, \eqref{itm:uvAdjacent_5common_neighbours_C3+2adjacent_nonedges}, and~\eqref{itm:uvAdjacent_5common_neighbours_independent_3-set}  of Lemma~\ref{lem:adj_5_common_nb}. 
\end{proof}

\begin{proof} [Proof of Lemma~\ref{lem:uv_adj_4common_nb_forbidden_structures}] 
If $G[N]$ contains a vertex of degree at least $3$ we are done by applying Observation~\ref{obs:no_deg3}
as well as Lemma~\ref{lem:adj_4_common_nb}.\eqref{itm:uvAdjacent_4common_neighbours_P4} and~\eqref{itm:uvAdjacent_4common_neighbours_independent_3-set}. 
Thus,  $G[N]$ must be isomorphic to one of the graphs that we will treat now.

%
	
	First suppose that the edge set of $G[N]$ equals the empty set, $\lbrace x_1x_2 \rbrace$ or  $\lbrace x_1x_2 , x_1x_3\rbrace$.
	Then, $G$ contains an independent $3$-set and has a legal colouring by Lemma~\ref{lem:adj_4_common_nb}.\eqref{itm:uvAdjacent_4common_neighbours_independent_3-set}.
	
	Next suppose that the edge set of $G[N]$ is equal to  $\lbrace x_1x_2, x_3x_4 \rbrace $ or to $\lbrace x_1x_2, x_1x_3, x_3x_4 \rbrace $. If $x_u$ is adjacent to $x_2$, apply Lemma~\ref{lem:adj_4_common_nb}.\eqref{itm:uvAdjacent_4common_neighbours_some_edges_there_some_not} to get a legal colouring: the edges $x_ux_2, x_2x_1,x_3x_4$ exist while $x_4$ is neither adjacent to $x_1$ nor to $x_2$. Similarly,  we can apply Lemma~\ref{lem:adj_4_common_nb}.\eqref{itm:uvAdjacent_4common_neighbours_some_edges_there_some_not} if $x_3x_v\in E(G)$. Thus, we can suppose that neither $x_2x_u$ nor $x_3x_v$ exists in $G$ and we are done with Lemma~\ref{lem:adj_4_common_nb}.\eqref{itm:uvAdjacent_4common_neighbours_notP4+edge}: the edges $x_ux_2, x_2x_3,x_3x_v$ do not exist while $x_2x_1 \in E(G)$.

Now suppose that the edge set of $G[N]$ consists of $x_1x_2, x_2x_3, x_3x_1$. If $x_u$ is adjacent to $x_1$, not all paths of length $3$ can be monochromatic and we can apply Lemma~\ref{lem:adj_4_common_nb}.\eqref{itm:uvAdjacent_4common_neighbours_P4}. Thus we can suppose that 
$x_ux_1 \notin E(G). $
If $x_4x_v \notin E(G)$ then we can apply Lemma~\ref{lem:adj_4_common_nb}.\eqref{itm:uvAdjacent_4common_neighbours_notP4+edge} to $x_ux_1,x_1x_4,x_4x_v \notin E(G)$ and $x_1x_3 \in E(G)$ to obtain a legal colouring of $G$. If $x_4x_v\in E(G)$ we are done by Lemma~\ref{lem:adj_4_common_nb}.\eqref{itm:uvAdjacent_4common_neighbours_C3+2adjacent_nonedges}.
		
	Last suppose that the edge set of $G[N]$ consists of $x_1x_2, x_2x_3, x_3x_4, x_4x_1$. If the $4$-cycle is not monochromatic, the cycle contains a $P_4$ that is not monochromatic and we are done by Lemma~\ref{lem:adj_4_common_nb}.\eqref{itm:uvAdjacent_4common_neighbours_P4}.
	Suppose that  $x_1x_u$ is an edge of $G$. Then, $x_ux_1x_2x_3$ is a $P_4$ that is not monochromatic. By symmetry, we get that neither $x_u$ nor $x_v$ is adjacent to a vertex of $N$. But then apply Lemma~\ref{lem:adj_4_common_nb}.\eqref{itm:uvAdjacent_4common_neighbours_notP4+edge} to  $x_ux_1, x_1x_3,x_3x_v \notin E(G)$ and $x_1x_2 \in E(G)$ to obtain a legal colouring of $G$. 
\end{proof}

\begin{proof}[Proof of Lemma~\ref{lem:nonadj_6_common_nb_forbidden}] 
If $G[N]$ contains a vertex of degree at least $3$ we are done  by applying Observation~\ref{obs:no_deg3}
as well as Lemma~\ref{lem:nonadj_6_common_nb}.\eqref{itm:uv_not_adj_6common_neighbours_indep_3set} and~\ref{lem:nonadj_6_common_nb}.\eqref{itm:uv_not_adj_6common_neighbours_P4}.

Now, suppose that $G[N]$ contains a vertex, say $x_1$ of degree $0$.
	As we have seen, $G[N]$ contains no vertex of degree at least $3$. Thus, $G[N]-\lbrace x_1 \rbrace$ contains two non-adjacent vertices, say $x_2$ and $x_3$. Then, $\lbrace x_1,x_2,x_3 \rbrace$ is an independent set and we are done by Lemma~\ref{lem:nonadj_6_common_nb}.\eqref{itm:uv_not_adj_6common_neighbours_indep_3set}.
	
We can conclude that all vertices in $G[N]$ have degree $1$ or $2$. Thus, $G[N]$ is isomorphic to one of the following graphs: $C_3 \dot{\cup} C_3$, $C_6$, $C_4 \dot{\cup} P_2$, $C_3 \dot{\cup} P_3$, $P_3 \dot{\cup} P_3$, $P_4 \dot{\cup} P_2$, $P_2 \dot{\cup} P_2 \dot{\cup} P_2$. 

If $G[N]$ is isomorphic to $C_3 \dot{\cup} C_3$, we can apply Lemma~\ref{lem:nonadj_6_common_nb}.\eqref{itm:uv_not_adj_6common_neighbours_P3_P3}. It is not possible that all pairs of $3$-paths
have three edges of the same colour.
In all other cases, we can apply Lemma~\ref{lem:nonadj_6_common_nb}.\eqref{itm:uv_not_adj_6common_neighbours_indep_3set}. 	
\end{proof}

\begin{proof} [Proof of Lemma~\ref{lem:deg2_or_4_neighbourhood_odd}]
The proof is based on the following observation: a legal colou\-ring $c'$ of $G'$ consists of at most $ \left\lfloor \sfrac{(|V(G)|-2)}{2} \right\rfloor = \lfloor\sfrac{(|V(G)|-3)}{2}\rfloor$ colours while a legal colouring of $G$ can consist of $\lfloor\sfrac{(|V(G)|-3)}{2}\rfloor+1=\left\lfloor \sfrac{(|V(G)|-1)}{2} \right\rfloor$ many colours.
We will now consider the neighbourhood of $u$ in $G$. 

If $u$ has exactly two neighbours $x_1$ and $x_2$ that are non-adjacent, set $G'=G-\lbrace u \rbrace + \lbrace x_1x_2 \rbrace $ and set $c(x_1ux_2)=c'(x_1x_2)$.\\
If $u$ has exactly two neighbours $x_1$ and $x_2$ that are adjacent, set $G'=G-\lbrace u \rbrace - \lbrace x_1x_2 \rbrace $ and set $c(x_1ux_2x_1)=\left\lfloor \sfrac{(|V(G)|-1)}{2} \right\rfloor$.
Further, set $c(e)=c'(e)$ for all other edges in both cases to obtain a legal colouring. 

If $u$ has exactly four neighbours $x_1, \ldots, x_4$ such that $x_1x_2, x_3x_4 \notin E(G)$ set $G'=G- \lbrace u \rbrace + \lbrace  x_1x_2, x_3x_4 \rbrace$ and set $c(x_1ux_2)=c'(x_1x_2)$ and $c(x_3ux_4)=c'(x_3x_4)$. If $c'(x_1x_2) \not= c'(x_3x_4)$, setting $c(e)=c'(e)$ for all other edges gives a legal colouring. If $c'(x_1x_2) = c'(x_3x_4)$, we again set $c(e)=c'(e)$ for all other edges. Now, $c$ is a colouring of $G$ where one colour class consists of two cycles intersecting only at $u$. We can split up this colour class into two cycles to obtain a legal colouring of $G$.

By Lemma~\ref{lem:nbhd_deg4_vertex}, we are done unless $u$ has four neighbours $x_1, x_2,x_3, x_4$ that form a clique. 
In that case, set $G'=G - \lbrace u \rbrace  - \lbrace x_1x_3, x_2x_4 \rbrace $. \\
If $x_1x_2$ and $x_3x_4$ are of different colour, set 
\begin{equation*}
c(x_1ux_2)=c'(x_1x_2) \quad c(x_3ux_4)=c'(x_3x_4) \quad c(x_1x_2x_4x_3x_1)= \left\lfloor \sfrac{(|V(G)|-1)}{2} \right\rfloor.
\end{equation*}
If the cycle $x_1x_2x_3x_4x_1$ is monochromatic , set 
\begin{equation*}
c(x_1ux_2x_3x_4x_1)=c'(x_1x_2x_3x_4x_1) \quad c(x_4ux_3x_1x_2x_4) = \left\lfloor \sfrac{(|V(G)|-1)}{2} \right\rfloor.
\end{equation*}
If the cycle is not monochromatic but $x_1x_2$ and $x_3x_4$ are of the same colour $i$ and $x_1x_4$ and $x_2x_3$ are of the same colour $j$, we need to distinguish two cases. By Observation~\ref{obs:one_cycle_two_cycles}, there are two ways for the shape of the cycle $C_i=c'^{-1}(i)$ with colour $i$. If $C_i- \lbrace x_1x_2, x_3x_4 \rbrace $ consists of a $x_1$-$x_3$-path and a $x_2$-$x_4$-path, set 
\begin{equation*}
 c(x_1ux_4)= i \quad  c(x_2x_3)=i \quad c(x_2ux_3)=j \quad c(x_1x_2x_4x_3x_1)= \left\lfloor \sfrac{(|V(G)|-1)}{2} \right\rfloor. 
\end{equation*}
If $C_i- \lbrace x_1x_2, x_3x_4 \rbrace $ consists of a $x_4$-$x_1$-path $P_{41}$ and a $x_2$-$x_3$-path, set
\begin{equation*}
c(x_1ux_4)= i \quad c(x_2ux_3)=j \quad c(x_1x_3x_2x_4P_{41}x_1)= \left\lfloor \sfrac{(|V(G)|-1)}{2} \right\rfloor. 
\end{equation*}
By setting $c(e)=c'(e)$ for all other edges $e$, $c$ is a legal colouring of $G$. 
\end{proof}

\begin{proof}[Proof of Lemma~\ref{lem:deg6_neighbourhood}] 
We transform the legal colouring $c'$ of $G'$ into a legal colou\-ring $c$ of $G$. 
For this, we first note that $u$ has degree $4$ in $G'$, ie $\{x_1,x_2,x_3,x_4\}$ splits up into two pairs $\lbrace a,a' \rbrace $ and $ \lbrace b,b' \rbrace $  with $c'(ua)=c'(ua')$ and $c'(ub)=c'(ub')$ and $c'(ua)\neq c'(ub)$.

If the colour $c'(x_5x_6)$ is not incident with $u$ in $G'$, set $c(x_5ux_6)=c'(x_5x_6)$ and leave all other edge colours untouched to get a legal colouring.

Now suppose that $c'(x_5x_6)$ is incident with $u$ (say $c'(aua')=c'(x_5x_6)$),  but the set $\{c'(ua), c'(aa'), c'(ub), c'(bb'), c'(x_5x_6)\}$  consists of at least three different colours.  
Then, there are two possible constellations. 
First, let $c'(aa')\neq c'(x_5x_6)$ and $c'(aa')\neq c'(bub')$. Then, set $c(x_5ux_6)=c'(x_5x_6)$, flip the colours of the edges $aua'$ and $aa'$ and leave all other edge colours untouched to get a legal colouring. 

If $c'(aa')= c'(bub')$ and $c'(bb')\neq c'(x_5x_6)$,  set $c(x_5ux_6)=c'(x_5x_6)$, flip the colours of the edges $aua'$ and $aa'$ and the colours of the edges $bub'$ and $bb'$, and leave all other edge colours untouched to get a legal colouring.

Thus, without loss of generality $c'(x_5x_6)=c'(aua')=c'(bb')$ and $c'(aa')=c'(bub')$. That is, among the considered edges there are only two colours.
We may assume that $c'(a'b)\neq c'(x_5x_6)$ and $c'(a'b')\neq c'(x_5x_6)$. Because, if eg $c'(a'b)= c'(x_5x_6)$, then $c'(ab)\neq c'(x_5x_6)$ and $c'(ab')\neq c'(x_5x_6)$. This is symmetric to the assumption. 

If $c'(a'b)=c'(bub')$ and $c'(a'b')\neq c'(bub')$ 
the following is a legal colouring for $G$:
\begin{gather*}
c(aa')=c'(aua')  \qquad  c(a'ub')=c'(a'b')\\
c(auba'b')=c'(aa'bub')  \qquad  c(x_5ux_6)=c'(x_5x_6)\\
c(e)=c'(e) \text{ for all other edges } e
\end{gather*}

If $c'(a'b')=c'(bub'), \, c'(a'b)\neq c'(bub')$, the following colouring for $G$ is legal:
\begin{gather*}
c(aa')=c'(aua')  \qquad  c(a'ub)=c'(a'b)\\
c(aub'a'b)=c'(aa'b'ub)  \qquad  c(x_5ux_6)=c'(x_5x_6)\\
c(e)=c'(e) \text{ for all other edges } e
\end{gather*}

Otherwise by Observation~\ref{obs:one_cycle_two_cycles}, one of the following is a legal colouring for $G$:
\begin{gather*}
c_1(aub')=c'(aa')  \qquad  c_1(a'b)=c'(aa')\\
c_1(a'ub)=c'(a'b)  \qquad  c_1(aa')=c'(aua')\\
c_1(x_5ux_6)=c'(x_5x_6)\\
c_1(e)=c'(e) \text{ for all other edges } e
\end{gather*}
\begin{center}
 or
\end{center}  
\begin{gather*}
c_2(aub)=c'(aa')  \qquad  c_2(a'b')=c'(aa')\\
c_2(a'ub')=c'(a'b')  \qquad  c_2(aa')=c'(aua')\\
c_2(x_5ux_6)=c'(x_5x_6)\\
c_2(e)=c'(e) \text{ for all other edges } e \qedhere
\end{gather*}
\end{proof}

\section{Path-decompositions} \label{sec:path-decomp} 

 For a graph $G$ a \emph{path-decomposition} $(\mathcal{P},\mathcal{B})$ consists of a path $\mathcal{P}$ and a collection  $\mathcal{B} = \{B_t \colon t \in
V(\mathcal{P}) \}$ of \emph{bags} $B_t \subset V(G)$ such that 
 \begin{itemize}
  \item $V(G) = \bigcup\limits_{t \in V(\mathcal{P})} B_t,$
  \item for each edge $vw \in E(G)$ there exists a vertex $t \in V(\mathcal{P})$ such that
$v,w \in B_t$, and
  \item if $v \in B_{s} \cap B_{t},$ then $v \in B_r$ for each vertex
$r$  on the path connecting $s$ and $t$ in $\mathcal{P}$.
\end{itemize}
A path-decomposition $(\mathcal{P},\mathcal{B})$ has \emph{width} $k$ if
each bag has a size of at most $k+1$. The \emph{pathwidth} of $G$ is
the smallest integer $k$ for which there is a width $k$
path-decomposition of $G$. 

In this paper, all paths $\mathcal{P}$ have vertex set $\lbrace 1, \ldots, n' \rbrace$ and edge set $\lbrace  ii+1 \, : i \in \lbrace  1, \ldots, n'-1 \rbrace \rbrace$.
A path-decomposition $(P,\mathcal{B})$ of
width $k$ is \emph{smooth} if
\begin{itemize} 
  \item  $ |B_i| = k+1$ for all $i \in \lbrace  1, \ldots, n' \rbrace$ and
  \item  $|B_i \cap B_{i+1}| = k$ for all $i \in \lbrace  1, \ldots, n'-1 \rbrace$.
 \end{itemize}
A graph of pathwidth at most $k$ always has a smooth path-decomposition of width~$k$; see Bodlaender~\cite{Bodlaender98}. Note that this path-decomposition has exactly $n'=|V(G)|-k$ many bags. 

If $(\mathcal{P},\mathcal{B})$ is a path-decomposition of the graph $G$, then 
for any vertex set $W$ of $G$
we denote by $\mathcal{P}(W)$ the subpath of $\mathcal{P}$ that consists of those bags
that contain a vertex of~$W$. 
Further, if $\mathcal{P}(W)$ is the path on vertex set $\lbrace  s, s +1, \ldots, t-1, t\rbrace$ with $s \leq t$ we denote $s$ by $s(W)$ and $t$ by $t(W)$. 
For $W= \lbrace v \rbrace$, we abuse notation and denote $\mathcal{P}(W)$, $s(W)$ and $t(W)$ by $\mathcal{P}(v), s(v)$ and $t(v)$. 

We note: in a smooth path-decomposition, for an edge $st\in E(\mathcal{P})$,
there is exactly one vertex $v \in V(G)$ with $v \in
B_s$ and $v \notin B_t$. We call this vertex~$v(s,t)$.
Thus for any vertex~$v$ of $G$, the number of vertices
in  the union of all bags 
containing $v$ is at most~$|\mathcal{P}(v )|+k$ and 
\begin{equation} \label{eq:deg_bound}
\deg(v)\leq |\mathcal{P}(v)|+k-1. 
\end{equation}
Consequently, 
\begin{gather}
\text{ for every } i \in \lbrace1, \ldots, \lfloor\sfrac{|V(G)|}{2}\rfloor \rbrace \text{ there are unique vertices } \nonumber\\
 v(i,i+1) \text{ and } v(n'+1-i, n'-i) \text{ with degree at most } i \label{eq:deg_bound_more_exact}
\end{gather}

For the proof of Theorem~\ref{thm:mainthm_pw6} we last note a direct consequence of a theorem of Fan and Xu~\cite[Theorem 1.1]{Fan_Xu02}:
\begin{corollary}\label{cor:FanXu}
Let $G$ be an Eulerian graph of pathwidth at most $6$ that does not sa\-tis\-fy Haj\'os' conjecture.  Then there is a graph $G'$ of pathwidth at most $6$ with $|V(G')| \leq |V(G)|$  that contains at most one vertex of degree less than $6$ and does not satisfy Haj\'os' conjecture.
\end{corollary}

Based on~\eqref{eq:deg_bound_more_exact} and on Lemma~\ref{lem:uv_adj_5common_nb_forbidden_structures},
\ref{lem:uv_adj_4common_nb_forbidden_structures} and~\ref{lem:nonadj_6_common_nb_forbidden} we finally show that Hajós' conjecture is satisfied for all Eulerian graphs of pathwidth $6$.

\begin{proof} [Proof of Theorem~\ref{thm:mainthm_pw6}]
Assume that the class of  graphs with pathwidth at most $6$ does not satisfy Haj\'os' conjecture. Let $G$ be a member of the class that does not satisfy the conjecture with minimal number of vertices.
By Theorem~\ref{thm:Hajos_12vertices},  $G$ has at least $13$ vertices. 
By Corollary~\ref{cor:FanXu}
\begin{equation} \label{eq:thmproof_deg_less_than_6_once}
\text{ $G$ contains at most one vertex of degree $2$ or $4$. } 
\end{equation} 
Thus, the three vertices $v(i,i+1)$ with $i=1,2,3$ or the three vertices $v(i,i-1)$ with $i=n',n'-1,n'-2$  all have degree at least $6$. Without loss of generality, suppose that 
\begin{equation} \label{eq:degree_at_least_6}
\deg(u),\deg(v),\deg(w)\geq 6 \text{ for } u \coloneqq v(1,2), v \coloneqq v(2,3) , w \coloneqq v(3,4)
\end{equation}
As $u$ and $v$ are both of degree $6$, there are three possibilities. 
\begin{enumerate} [ \normalfont \rmfamily (i)]
\item $u$ and $v$ have common neighbourhood $N=\{x_1,\ldots,x_6\}$, or \label{itm:nonadj-6common_nb}
\item $u$ and $v$ are adjacent with common neighbourhood $N=\{x_1,\ldots,x_5\}$, or
\label{itm:adj-5common_nb}
 \item $u$ and $v$ are adjacent with common neighbourhood $N=\{x_1,\ldots,x_4\}$ and private neighbours $x_u$ and $x_v$. 
 \label{itm:adj-4common_nb}
\end{enumerate}
We will now always delete $u$ and $v$ and optionally some edges. Further, we optionally add some edges in the neighbourhood of the two vertices. The obtained graph is still of pathwidth $6$ and consequently has a legal colouring.

First assume~\eqref{itm:nonadj-6common_nb} or~\eqref{itm:adj-5common_nb}. By Lemma~\ref{lem:nonadj_6_common_nb_forbidden}  and  Lemma~\ref{lem:uv_adj_5common_nb_forbidden_structures}, $N$ is an independent set and there is no vertex in $G - \lbrace u, v \rbrace $ that has at least three neighbours in $N$. This is not possible as $w$ must have at least six neighbours in $ N \cup \lbrace  u,v \rbrace$ by~\eqref{eq:degree_at_least_6}. 

\bigskip

Last assume~\eqref{itm:adj-4common_nb} and define $u'=v(n',n'-1)$, $v'=v(n'-1,n'-2)$ and $w'=v(n'-2,n'-3)$. 
By symmetry of the two sides of the path $\mathcal{P}$ of $G$'s path-decomposition, 
we can suppose that
\begin{enumerate} [\normalfont \rmfamily (I)]
\item \label{itm:thmproof_other_side_equal} $u'$ and $v'$ are two adjacent degree-$6$ vertices  with common neighbourhood $N'=\{x_1',\ldots,x_4'\}$ and private neighbours $x_u'$ and $x_v'$ and $\deg(w')\geq 6$,  or 
\item there is a vertex $y$ of degree less than $6$ among $u',v', w'$. 
\label{itm:thmproof_other_side_deg<4}
\end{enumerate}

Our aim is now to find a path between $x_u$ and $x_v$ in $G-R$ with $R=N \cup \{u,v\}$ (respectively a path between $x_{u'}$ and $x_{v'}$ in $G-R'$ with $R'=N' \cup \{u',v'\}$).
Then we can use Lemma~\ref{lem:uv_adj_4common_nb_forbidden_structures} to see that $N$ is an independent set and there is no vertex in $G - \lbrace u, v \rbrace $ that has at least three neighbours in $N$. This is not possible as $w$ must have at least six neighbours in $R$ by~\eqref{eq:degree_at_least_6}. 

We assume now that there is no path between $x_u$ and $x_v$ in $G-R$ with $R=N \cup \{u,v\}$ and denote the set of vertices in the component of $x_u$ in $G-R$ by $V_u$. Similarly, we define $V_v$. 
The vertex $z$ of $V_u$ respectively $V_v$ that maximises $s(z)$ is denoted by $z_u$ respectively $z_v$.
Note that the neighbourhood of $z_a$ (for $a=u$ and $a=v$ ) satisfies
\begin{equation} \label{eq:N(z_x)}
 N(z_a) \subseteq \left(B_{s(z_a)} \cap V_a \right). 
\end{equation}
First assume that $t(V_u)=t(V_v)$. Then, as $(P, \mathcal{B})$ is smooth, $t(V_u)=t(V_v)=n'$. By~\eqref{eq:N(z_x)}, the neighbours of $z_u$ and $z_v$ are contained in the sets $B_{s(z_u)} \cap V_u$ and $B_{s(z_v)} \cap V_v$ that are both of size at most $5$. This contradicts~\eqref{eq:thmproof_deg_less_than_6_once}. 
Thus we can suppose that 
\begin{equation} \label{eq:Va_ends_before_Vb}
t(V_u) < t(V_v) (\leq n').
\end{equation}
Then, $z_v$ might have degree $6$, but
\begin{equation*}
\mbox{ $z_u$  has degree less than $6$.}
\end{equation*}
Now, we split up the proof.
First assume that~\eqref{itm:thmproof_other_side_equal} holds.   
By symmetry, we can apply the previous part of the proof and find a vertex $z'_{u'}$ 
in the component $V'_{u'} \not= V'_{v'}$ in $G-R'$ (with $R'=\{u',v'\}\cup N'$).
By~\eqref{eq:thmproof_deg_less_than_6_once},  $z_u$ equals $z'_{u'}$. 

As $z_u \in V_u$, there is a path $P_{x_u,z_u}$ from $x_u$ to $z_u$ in $G-R$. 
Further, by~\eqref{eq:Va_ends_before_Vb}, there is no $x_u$-$x'_{u'}$-path $P_{x_u,x'_{u'}}$ in $G-R-R'$. 
This means that the path $P_{x_u, z_u} $ contains a vertex $r'$ of $R' \subset B_{n'}$ which contradicts~\eqref{eq:Va_ends_before_Vb}.

\medskip

Now assume that~\eqref{itm:thmproof_other_side_deg<4} holds. 
As we have seen before, we can assume~\eqref{eq:Va_ends_before_Vb} and obtain from~\eqref{eq:thmproof_deg_less_than_6_once}  that $y=z_u$

If $z_u=y=u'$ or $z_u=y=v'$, then all neighbours of $z_u$ (ie particularly a vertex of $V_u$) are contained in $B_{n'}$. This contradicts~\eqref{eq:Va_ends_before_Vb}.

Thus suppose that $z_u=y=w'$. Then $u'$ and $v'$ must be of degree $6$. If $u'$ and $v'$ have four common neighbours, then~\eqref{itm:thmproof_other_side_equal} holds and we are done. 
If $u'$ and $v'$ have five common neighbours then note that $v(n'-3, n'-4)$ must have degree at least $6$ by~\eqref{eq:thmproof_deg_less_than_6_once}. Thus we can apply Lemma~\ref{lem:uv_adj_5common_nb_forbidden_structures} to get a legal colouring of $G$. Therefore, it remains to consider that $u'$ and $v'$ are non-adjacent with common neighbourhood $N'$ of size $6$.  Then all neighbours of $z_u$ (ie a vertex of $V_u$) are contained in $N' \subset B_{n'}$ which  contradicts~\eqref{eq:Va_ends_before_Vb}.
\end{proof}

\bibliographystyle{abbrv}

\bibliography{cycle-decomp} 

\vfill

\small
\vskip2mm plus 1fill
\noindent
Version \today{}
\bigbreak

\noindent
Elke Fuchs
{\tt <elke.fuchs@uni-ulm.de>}\\
Laura Gellert
{\tt <laura.gellert@uni-ulm.de>}\\
Institut f\"ur Optimierung und Operations Research\\
Universit\"at Ulm, Ulm\\
Germany\\

\noindent
Irene Heinrich
{\tt <heinrich@mathematik.uni-kl.de>}\\
AG Optimierung\\
Technische Universit\"at Kaiserslautern, Kaiserslautern\\
Germany\\

\end{document}